\documentclass[english]{amsart}
\usepackage[T1]{fontenc}
\usepackage[latin9]{inputenc}
\usepackage{amsthm}
\usepackage{amssymb}

\makeatletter
\numberwithin{equation}{section}
\numberwithin{figure}{section}
\theoremstyle{plain}
\newtheorem{thm}{\protect\theoremname}
\theoremstyle{plain}
\newtheorem{lem}[thm]{\protect\lemmaname}
\theoremstyle{plain}
\newtheorem{cor}[thm]{\protect\corollaryname}
\theoremstyle{definition}
\newtheorem{example}[thm]{\protect\examplename}

\makeatother

\usepackage{babel}
\providecommand{\corollaryname}{Corollary}
\providecommand{\examplename}{Example}
\providecommand{\lemmaname}{Lemma}
\providecommand{\theoremname}{Theorem}

\begin{document}
\title{On the Zeros of Entire $q$-Functions}
\author{Ruiming Zhang}
\email{ruimingzhang@guet.edu.cn}
\address{School of Mathematics and Computing Sciences\\
Guilin University of Electronic Technology\\
Guilin, Guangxi 541004, P. R. China. }
\thanks{This work is supported by the National Natural Science Foundation
of China grants No. 11771355 and No. 12161022.}
\begin{abstract}
In this work we prove that entire $q$-functions have infinitely many
nonzero roots $\left\{ \rho_{n}\right\} _{n=1}^{\infty}$, as $n\to+\infty$
the moduli $\left|\rho_{n}\right|$ grow at least exponentially. Applications
to $q$-transcendental entire functions defined by series expansions
are provided. These functions include the $q$-analogue of the plane
wave function $\mathcal{E}_{q}(z,t)$. 
\end{abstract}

\maketitle

\section{\label{sec:Intro} Introduction }

Many well-known entire $q$-functions are of type 
\begin{equation}
f(z)=\sum_{n=0}^{\infty}f_{n}q^{\alpha n^{2}}z^{n},\quad\sup_{n\ge0}\left|f_{n}\right|<\infty,\ \sup\left\{ n:f_{n}\neq0\right\} =\infty,\label{eq:1.1}
\end{equation}
where $\alpha>0$ and $0<\left|q\right|<1$. They are order $0$ entire
functions with infinitely many nonzero roots $\left\{ \rho_{n}\right\} _{n=1}^{\infty}$
with moduli satisfying
\begin{equation}
0<\left|\rho_{1}\right|\le\left|\rho_{2}\right|\le\cdots.\label{eq:1.2}
\end{equation}
For example, a Euler $q$-exponential function $E_{q}(z)$, \cite{AndrewsAskeyRoy,IsmailBook}
\begin{equation}
E_{q}(z)=\sum_{n=0}^{\infty}\frac{q^{\binom{n}{2}}z{}^{n}}{(q;q)_{n}},\quad z\in\mathbb{C}\label{eq:1.3}
\end{equation}
and Ramanujan's entire function $A_{q}(z)$, \cite{Andrews1,IsmailBook}
\begin{equation}
A_{q}(z)=\sum_{n=0}^{\infty}\frac{q^{n^{2}}(-z)^{n}}{(q;q)_{n}},\quad z\in\mathbb{C}.\label{eq:1.4}
\end{equation}
Here the $q$-shifted factorials are defined by \cite{AndrewsAskeyRoy,IsmailBook,KoekoekSwarttouw}
\begin{equation}
(a;q)_{\infty}=E_{q}(-z),\quad(a;q)_{n}=\frac{(a;q)_{\infty}}{(aq^{n};q)_{\infty}},\quad a,n\in\mathbb{C},\ \left|q\right|<1.\label{eq:1.5}
\end{equation}
More generally, for any nonnegative integers $r,s$ with $s\ge r$
the basic hypergeometric seres (a.k.a $q$-series), \cite{AndrewsAskeyRoy,IsmailBook,KoekoekSwarttouw}
\begin{equation}
_{r}\phi_{s}\left(\begin{array}{c}
a_{1},\dots,a_{r}\\
b_{1},\dots,b_{s}
\end{array}\bigg|q,z\right)=\sum_{n=0}^{\infty}\frac{(a_{1},a_{2},\dots,a_{r};q)_{n}z^{n}}{(q,b_{1},b_{2},\dots,b_{s};q)_{n}}\left(-q^{\frac{n-1}{2}}\right)^{n(s+1-r)}\label{eq:1.6}
\end{equation}
defines an entire function in variable $z$ where $a_{1},\ a_{2},\ \dots,\ a_{r},\ b_{1},\ b_{2},\ \dots,\ b_{s}$
are complex numbers and the short-hand notation $(a_{1},a_{2},\dots,a_{r};q)_{n}$
stands for
\begin{equation}
(a_{1},a_{2},\dots,a_{r};q)_{n}=\prod_{k=1}^{r}(a_{k};q)_{n}.\label{eq:1.7}
\end{equation}
Evidently, the functions $E_{q}(z)$, $A_{q}(z)$, the $q$-Bessel
functions $J_{\nu}^{(2)}(z;q)z^{-\nu}$ and $J_{\nu}^{(3)}(z;q)z^{-\nu}$
, \cite{Hayman,IsmailBook,Koelink,KoekoekSwarttouw}
\begin{align}
J_{\nu}^{(2)}(z;q) & =\frac{(q^{\nu+1};q)_{\infty}}{(q;q)_{\infty}}\left(\frac{z}{2}\right)^{\nu}\sum_{n=0}^{\infty}\frac{q^{n^{2}+n\nu}\left(-z^{2}/4\right)^{n}}{(q,q^{\nu+1};q)_{n}},\label{eq:1.8}\\
J_{\nu}^{(3)}(z;q) & =\frac{(q^{\nu+1};q)_{\infty}}{(q;q)_{\infty}}\left(\frac{z}{2}\right)^{\nu}\sum_{n=0}^{\infty}\frac{q^{n(n+1)/2}\left(-z^{2}/4\right)^{n}}{(q,q^{\nu+1};q)_{n}},\label{eq:1.9}
\end{align}
are all special cases of (\ref{eq:1.6}). For arbitrary positive $\alpha$
a more general $q$-exponential function $E_{q}^{(\alpha)}(z;q)$
is defined by 
\begin{equation}
E_{q}^{(\alpha)}(z;q)=\sum_{n=0}^{\infty}\frac{q^{\alpha n^{2}}z^{n}}{(q;q)_{n}},\quad z\in\mathbb{C},\label{eq:1.10}
\end{equation}
it is clearly of type (\ref{eq:1.1}) but not of type (\ref{eq:1.6}). 

We notice that all the functions in (\ref{eq:1.6}) and (\ref{eq:1.10})
satisfy the condition 
\begin{equation}
f_{n}=A\left(1+\mathcal{O}\left(q^{n}\right)\right),\quad n\to\infty,\label{eq:1.11}
\end{equation}
where $A$ is a nonzero constant. Under condition (\ref{eq:1.11})
Hayman proved that the $n$-th nonzero root $\rho_{n}$ of $f(z)$
in (\ref{eq:1.1}) possess a asymptotic expansion, \cite{Hayman,Ismail1}
\begin{equation}
\rho_{n}=q^{1-2n}\sum_{k=0}^{\infty}d_{k}q^{nk},\quad n\to\infty,\label{eq:1.12}
\end{equation}
where $\left\{ d_{k}\right\} _{k=0}^{\infty}$ are constants independent
of $n$ with $d_{0}\neq0$. 

However, there are many entire functions are of type (\ref{eq:1.1})
not satisfying (\ref{eq:1.11}). For example, 
\begin{equation}
\sum_{n=0}^{\infty}f_{n}q^{a(n)}b_{n}^{it}z^{n},\quad\sup_{n\ge0}\left|f_{n}\right|<\infty,\ \sup\left\{ n:f_{n}\neq0\right\} =\infty,\label{eq:1.13}
\end{equation}
where 
\begin{equation}
a(n)=a_{0}+a_{1}n+\cdots+a_{k}n^{k},\quad k\ge2,\ a_{k}>0,\quad b_{n}>0,\quad t\in\mathbb{R}.\label{eq:1.14}
\end{equation}
There are many entire functions defined as polynomial expansions.
For example, 
\begin{equation}
\mathcal{L}(z;\alpha,q)=\sum_{n=0}^{\infty}\left(q^{\alpha+2n+\frac{1}{2}};q\right)_{\infty}L_{n}^{(\alpha+n-\frac{1}{2})}\left(z;q\right)q^{n^{2}/2}q^{\alpha n}\label{eq:1.16}
\end{equation}
is such a example. Here $L_{n}^{(\alpha)}\left(z;q\right)$ is the
$q$-Laguerre polynomials respectively. The $q$-exponential function
$\mathcal{E}_{q}(z;t)$ of Ismail-Zhang provides yet another example,
\cite{IsmailBook,IsmailRZhang1,IsmailStanton}. It can be defined
by \cite{IsmailBook,IsmailStanton,IsmailRZhang1}
\begin{equation}
(qt^{2};q^{2})_{\infty}\mathcal{E}_{q}(z;t)=\sum_{n=0}^{\infty}\frac{q^{\frac{n^{2}}{4}}t^{n}}{(q;q)_{n}}H_{n}(z\vert q),\quad\left|t\right|<1,\label{eq:1.17}
\end{equation}
 where the $q$-Hermite polynomials are defined by
\begin{equation}
\frac{1}{\left(te^{i\theta},te^{-i\theta};q\right)_{\infty}}=\sum_{n=0}^{\infty}\frac{H_{n}(\cos\theta\vert q)t^{n}}{(q;q)_{n}},\quad\left|t\right|<1,\ \theta\in[0,\pi].\label{eq:1.18}
\end{equation}
For any fixed $q,\ t$ with $0<q<1$ and $\left|t\right|<1$, it is
known that the function $\mathcal{E}_{q}(z;t)$ is entire in $z$
satisfying\cite{IsmailBook,IsmailStanton}
\begin{equation}
\lim_{q\uparrow1}\mathcal{E}_{q}\left(x;\frac{(1-q)t}{2}\right)=e^{tx}\label{eq:1.19}
\end{equation}
 and
\begin{equation}
\lim_{r\to\infty}\frac{\log M(r,\mathcal{E}_{q})}{\log^{2}r}=\frac{1}{\log q^{-1}},\label{eq:1.20}
\end{equation}
 where for any $r>0$
\begin{equation}
M(r,f)=\sup\left\{ \left|f(z)\right|:\left|z\right|\le r\right\} .\label{eq:1.21}
\end{equation}
In this work we prove if there exists a positive number $A>0$ such
that
\begin{equation}
\log M(r,f)=\mathcal{O}\left(\log^{A}r\right),\quad r\to\infty\label{eq:1.22}
\end{equation}
then the entire function $f(z)$ has infinitely many nonzero roots
$\left\{ \rho_{n}\right\} _{n=1}^{\infty}$, as $n\to\infty$ $\log\left|\rho_{n}\right|$
grow at least as fast as a linear function of $n$ for $A\ge1$, whereas
for $1>A>0$ $\log\left|\rho_{n}\right|$ grow at least as a linear
function of $n^{1/A}$ .

\section{Main Results \label{sec:Main-Results}}

Under the additional conditions $0<q<1$ the class of functions defined
in (\ref{eq:1.1}) is equivalent to the class of entire functions
described by Ismail in \cite[Lemma 14.1.4]{IsmailBook} via a obvious
scaling transformation $q^{\alpha}\to p$,
\begin{equation}
F(z)=\sum_{n=0}^{\infty}f_{n}p^{n^{2}}z^{n},\quad0<p<1.\label{eq:2.1}
\end{equation}
It is shown there that $F(z)$ is order $0$, it has infinitely many
zeros and satisfies 
\begin{equation}
\limsup_{r\to\infty}\frac{\log M(r,F)}{\log^{2}r}\le\frac{1}{4\log p^{-1}}.\label{eq:2.2}
\end{equation}
The following lemma is a minor refinement of Lemma 14.1.4 in \cite{IsmailBook},
it reduces the functions of type (\ref{eq:1.1}) to the cases covered
by Theorem \ref{thm:series}.
\begin{lem}
\label{lem:modulus} If 
\begin{equation}
f(z)=\sum_{n=0}^{\infty}f_{n}q^{\alpha n^{2}}z^{n},\quad\left\{ f_{n}\right\} _{n=0}^{\infty}\subset\mathbb{C}\label{eq:2.3}
\end{equation}
 such that 
\begin{equation}
\alpha>0,\quad0<\left|q\right|<1,\quad\sup_{n\ge0}\left|f_{n}\right|<\infty,\label{eq:2.4}
\end{equation}
then for any $r>0$,
\begin{equation}
M(r,f)\le\begin{cases}
\frac{\theta\left(\left|q\right|^{2\alpha}\right)\sup_{n\ge0}\left|f_{n}\right|}{\left|q\right|^{\alpha}}, & 0\le r\le\left|q\right|^{-2\alpha},\\
2\theta\left(\left|q\right|^{2\alpha}\right)\sup_{n\ge0}\left|f_{n}\right|\cdot r\cdot\exp\left(-\frac{\log^{2}r}{4\alpha\log\left|q\right|}\right), & r>\left|q\right|^{-2\alpha},
\end{cases}\label{eq:2.5}
\end{equation}
where $\theta(q)=\sum_{n\in\mathbb{Z}}q^{n^{2}/2}$. Consequently,
\begin{equation}
\limsup_{r\to\infty}\frac{\log M(r,f)}{\log^{2}r}\le-\frac{1}{4\alpha\log\left|q\right|}.\label{eq:2.6}
\end{equation}
\end{lem}

\begin{proof}
It is clear that by obvious scalings (\ref{eq:2.3}) can be transformed
into the form, 
\begin{equation}
f(z)=\sum_{n=0}^{\infty}f_{n}q^{n^{2}/2}z^{n},\quad0<\left|q\right|<1,\ \sup_{n\ge0}\left|f_{n}\right|<\infty.\label{eq:2.7}
\end{equation}
For any $r$ with $0<r\le\left|q\right|^{-1}$ on $\left|z\right|\le r$
we have
\begin{equation}
\left|f(z)\right|\le\sup_{n\ge0}\left|f_{n}\right|\cdot\sum_{n=0}^{\infty}\left|q\right|^{n^{2}/2}r^{n}\le\sup_{n\ge0}\left|f_{n}\right|\cdot\sum_{n=0}^{\infty}\left|q\right|^{n^{2}/2-n}\le\frac{\theta\left(\left|q\right|\right)\sup_{n\ge0}\left|f_{n}\right|}{\sqrt{\left|q\right|}}.\label{eq:2.8}
\end{equation}
For $r>\left|q\right|^{-1}$ let 
\begin{equation}
\mu=\left\lfloor -\frac{\log r}{\log\left|q\right|}\right\rfloor \ge1,\quad1\le\left|q\right|^{\mu}r\le\frac{1}{\left|q\right|}.\label{eq:2.9}
\end{equation}
Then on $\left|z\right|\le r$ we have
\begin{equation}
\begin{aligned} & \left|f(z)\right|\le\sup_{n\ge0}\left|f_{n}\right|\cdot\left(\sum_{n=0}^{\mu}\left|q\right|^{n^{2}/2}r^{n}+\sum_{n=\mu+1}^{\infty}\left|q\right|^{n^{2}/2}r^{n}\right)\\
 & =\sup_{n\ge0}\left|f_{n}\right|\cdot\left(\sum_{k=0}^{\mu}\left|q\right|^{(\mu-k)^{2}/2}r^{\mu-k}+\sum_{k=1}^{\infty}\left|q\right|^{(\mu+k)^{2}/2}r^{\mu+k}\right)\\
 & =\sup_{n\ge0}\left|f_{n}\right|\cdot\left|q\right|^{\mu^{2}/2}r^{\mu}\left(\sum_{k=0}^{\mu}\left|q\right|^{k^{2}/2}\left(\left|q\right|^{\mu}r\right)^{-k}+\sum_{k=1}^{\infty}\left|q\right|^{k^{2}/2}\left(\left|q\right|^{\mu}r\right)^{k}\right)\\
 & \le\sup_{n\ge0}\left|f_{n}\right|\cdot\sqrt{\left|q\right|}\exp\left(-\frac{\log^{2}r}{2\log\left|q\right|}\right)\left(\sum_{k=0}^{\mu}\left|q\right|^{k^{2}/2}+\sum_{k=1}^{\infty}\left|q\right|^{k^{2}/2-k}\right)\\
 & \le2\theta\left(\left|q\right|\right)\sup_{n\ge0}\left|f_{n}\right|\cdot r\cdot\exp\left(-\frac{\log^{2}r}{2\log\left|q\right|}\right).
\end{aligned}
\label{eq:2.10}
\end{equation}
\end{proof}
Suppose that $r>0$ and $f(z)$ is an analytic function in a region
in the complex plane which contains the closed disk ${\displaystyle \left|z\right|\le r}$,
${\displaystyle \rho_{1},\rho_{2},\ldots,\rho_{n}}$ are the zeros
of $f(z)$ in the open disk ${\displaystyle \left|z\right|<r}$ (repeated
according to their respective multiplicity), and that ${\displaystyle f(0)=1}$.
Jensen's formula states that \cite{Ahlfors,Boas}

\begin{equation}
{\displaystyle \frac{1}{2\pi}\int_{0}^{2\pi}\log|f(re^{i\theta})|d\theta=\sum_{k=1}^{n}\log\left(\frac{r}{\left|\rho_{k}\right|}\right)=\int_{0}^{r}\frac{n(t)}{t}dt,}\label{eq:2.11}
\end{equation}
where $n(t)$ denotes the number of zeros of $f(z)$ in the disc of
radius $t$ centered at the origin.
\begin{thm}
\label{thm:series}Let $f(z)$ be an entire function such that $f(0)=1$
and for a positive number $A$, 
\begin{equation}
\limsup_{r\to\infty}\frac{\log M(r,f)}{\log^{A}r}<\infty.\label{eq:2.12}
\end{equation}
Then $f(z)$ has infinitely many nonzero roots $\left\{ \rho_{n}\right\} _{n\in\mathbb{N}}$
such that
\begin{equation}
0<\left|\rho_{1}\right|\le\left|\rho_{2}\right|\le\cdots\le\left|\rho_{n}\right|\le\cdots.\label{eq:2.13}
\end{equation}
Let $C$ be any positive number satisfying
\begin{equation}
C>\limsup_{r\to\infty}\frac{\log M(r,f)}{\log^{A}r}.\label{eq:2.14}
\end{equation}
Then we have: 
\begin{enumerate}
\item Case $0<A\le1$. For any $C$ in (\ref{eq:2.14}) there exists a positive
integer $N$ such that 
\begin{equation}
\left|\rho_{n}\right|>e^{2(2C)^{1/A}n^{1/A}},\quad\forall n\ge N.\label{eq:2.15}
\end{equation}
\item Case $A>1$. For any $C$ in (\ref{eq:2.14}) and any $\eta$ with
$1>\eta>0$ there exists a positive integer $N$ such that 
\begin{equation}
\left|\rho_{n}\right|\ge e^{C(\eta,A,C)n^{1/(A-1)}},\quad\forall n\ge N,\label{eq:2.16}
\end{equation}
where $C(\eta,A,C)$ is defined by,
\begin{equation}
C(\eta,A,C)=\left(\frac{(1-\eta)\eta^{A}}{C}\right)^{\frac{1}{A-1}}.\label{eq:2.17}
\end{equation}
\end{enumerate}
\end{thm}

\begin{proof}
By (\ref{eq:2.14}) the order $\rho(f)$ of the entire function function
$f(z)$, 
\begin{equation}
{\displaystyle \rho(f)=\limsup_{r\to\infty}\frac{\log M(r,f)}{\log r}=0.}\label{eq:2.18}
\end{equation}
Then by \cite[Theorem 1.2.5]{IsmailBook} we conclude that $f(z)$
has infinitely many zeros, and they can arranged in (\ref{eq:2.13})
since the root sequence has no limit points on the finite part of
the complex plan $\mathbb{C}$. 
\begin{enumerate}
\item Proof for the case $0<A\le1$. Let $n(R)$ denote the number of roots
$\rho_{n}$ inside the closed disk $\left|z\right|\le R$ (counted
with multiplicities). We first prove that for all sufficiently large
$R\ge2$,
\begin{equation}
n(R)\le\frac{C2^{A}}{\log2}\log^{A}R,\label{eq:2.19}
\end{equation}
where $C$ satisfies (\ref{eq:2.12}). Since for sufficiently large
$R\ge2$ we have 
\begin{equation}
\log M(2R,f)\le C\log^{A}2R\le C2^{A}\log^{A}R,\label{eq:2.20}
\end{equation}
then we apply Jensen's formula (\ref{eq:2.8}) to $f(z)$ on the closed
disk $|z|\le2R$ to get
\begin{equation}
\begin{aligned} & \sum_{\left|\rho\right|\le2R}\log\frac{2R}{\left|\rho\right|}\le\frac{1}{2\pi}\int_{0}^{2\pi}\log\left|f(2Re^{i\theta})\right|d\theta\\
 & \le\frac{1}{2\pi}\int_{0}^{2\pi}\log M(2R,f)d\theta=\log M(2R,f)\le C2^{A}\log^{A}R.
\end{aligned}
\label{eq:2.21}
\end{equation}
We break the left-hand sum into two sums according to $\left|\rho\right|\le R$,
keep the terms in the first sum then,
\begin{equation}
\begin{aligned} & C2^{A}\log^{A}R\ge\sum_{\left|\rho\right|\le2R}\log\frac{2R}{\left|\rho\right|}\\
 & \ge\sum_{\left|\rho\right|\le R}\log\frac{2R}{\left|\rho\right|}+\sum_{R<\left|\rho\right|\le2R}\log\frac{2R}{\left|\rho\right|}\ge n(R)\log2,
\end{aligned}
\label{eq:2.22}
\end{equation}
which proves (\ref{eq:2.19}). Let $\left\{ R_{n}\right\} _{n\in\mathbb{N}}$
be a sequence of positive numbers such that
\begin{equation}
C2^{A+1}\log^{A}R_{n}=n,\quad R_{n}=e^{2(2C)^{1/A}n^{1/A}},\label{eq:2.23}
\end{equation}
then,
\begin{equation}
n(R_{n})\le\frac{C2^{A}}{\log2}\log^{A}R_{n}=\frac{C2^{A}}{\log2}\frac{n}{C2^{A+1}}<\frac{10}{13}n.\label{eq:2.24}
\end{equation}
Thus, for all sufficiently large $n$ there are at most $\frac{10}{13}n<n$
roots (\ref{eq:2.10}) inside the closed disk $\left\{ z:\left|z\right|\le R_{n}\right\} $,
the $n$-th root $\rho_{n}$ must be outside the disk. Hence, $\left|\rho_{n}\right|>R_{n}=e^{2(2C)^{1/A}n^{1/A}}$. 
\item Proof for the case $A>1$. Given any positive number $\eta$ with
$0<\eta<1$ we observe that for any sufficiently large $\Delta>0$,
\begin{equation}
\int_{0}^{\Delta}\frac{n(t)}{t}dt\ge\int_{\Delta^{\eta}}^{\Delta}\frac{n(t)}{t}dt\ge n\left(\Delta^{\eta}\right)\int_{\Delta^{\eta}}^{\Delta}\frac{dt}{t}=n\left(\Delta^{\eta}\right)(1-\eta)\log\Delta.\label{eq:2.25}
\end{equation}
Then by (\ref{eq:2.11}) for any sufficiently large $\Delta>0$, 
\begin{equation}
n\left(\Delta^{\eta}\right)(1-\eta)\log\Delta\le\frac{1}{2\pi}\int_{0}^{2\pi}\log\left|f\left(\Delta e^{i\theta}\right)\right|d\theta\le C\log^{A}\Delta,\label{eq:2.26}
\end{equation}
 thus,
\begin{equation}
n\left(\Delta^{\eta}\right)\le\frac{C}{1-\eta}\log^{A-1}\Delta.\label{eq:2.27}
\end{equation}
 Let $R=\Delta^{\eta}$ then
\begin{equation}
n\left(R\right)\le\frac{C}{1-\eta}\log^{A-1}R^{1/\eta}=\frac{C\eta^{1-A}}{1-\eta}\log^{A-1}R.\label{eq:2.28}
\end{equation}
Let 
\begin{equation}
\frac{C}{(1-\eta)\eta^{A}}\log^{A-1}R_{n}=n,\quad R_{n}=\exp\left(C(\eta,A,C)n^{\frac{1}{A-1}}\right),\label{eq:2.29}
\end{equation}
where $C(\eta,A,C)$ is defined in (\ref{eq:2.17}), then 
\begin{equation}
n\left(R_{n}\right)\le\frac{C\eta^{1-A}}{1-\eta}\log^{A-1}R_{n}=\eta n<n.\label{eq:2.30}
\end{equation}
 Thus for any $C$ in (\ref{eq:2.14}) and $0<\eta<1$ there exists
positive integer $N$ such that
\begin{equation}
\left|\rho_{n}\right|\ge R_{n}=\exp\left(C(\eta,A,C)n^{1/(A-1)}\right),\quad\forall n\ge N.\label{eq:2.31}
\end{equation}
 
\end{enumerate}
\end{proof}
Apply Theorem \ref{thm:series} to Lemma \ref{lem:modulus} we obtain
the following result.
\begin{cor}
\label{cor:special-sums}Let $f(z)$ be the entire function defined
in (\ref{eq:2.3}), then it has infinitely many nonzero roots $\left\{ \rho_{n}\right\} _{n\in\mathbb{N}}$
satisfying (\ref{eq:2.13}) and for any number $C$ satisfying 
\begin{equation}
C>-\frac{1}{4\alpha\log\left|q\right|}>0,\label{eq:2.32}
\end{equation}
there exists a positive integer $N$ such that 
\begin{equation}
\left|\rho_{n}\right|\ge e^{C(\eta,C)n},\quad\forall n\ge N,\label{eq:2.33}
\end{equation}
where 
\begin{equation}
C(\eta,C)=\frac{(1-\eta)\eta^{2}}{C}.\label{eq:2.34}
\end{equation}
\end{cor}

\begin{proof}
The only complication comes from 
\begin{equation}
\nu=\inf\left\{ n:f_{n}\neq0\right\} >0.\label{eq:2.35}
\end{equation}
In this case we have
\begin{equation}
f(z)=\sum_{n=\nu}^{\infty}f_{n}q^{\alpha n^{2}}z^{n}=z^{\nu}\sum_{n=0}^{\infty}\left(f_{n+\nu}q^{2\alpha n\nu+\alpha\nu^{2}}\right)q^{\alpha n^{2}}z^{n}.\label{eq:2.36}
\end{equation}
Let
\begin{equation}
g(z)=\frac{f(z)z^{-\nu}}{f_{\nu}q^{\alpha\nu^{2}}}=\sum_{n=0}^{\infty}g_{n}q^{\alpha n^{2}}z^{n}\label{eq:2.37}
\end{equation}
with $g(0)=1$ and
\begin{equation}
\sup_{n\ge0}\left|g_{n}\right|=\sup_{n\ge0}\left|\frac{f_{n+\nu}q^{2\alpha n\nu+\alpha\nu^{2}}}{f_{\nu}q^{\alpha\nu^{2}}}\right|\le\frac{\sup_{n\ge0}\left|f_{n}\right|}{\left|f_{\nu}\right|}<\infty.\label{eq:2.38}
\end{equation}
We apply Theorem (\ref{thm:series}) with $A=2$ to $g(z)$ instead
of $f(z)$.
\end{proof}
\begin{example}
For $k>2,\ a_{k}>0,\ t\in\mathbb{R}\backslash\left\{ 0\right\} ,\ b_{n}>0$
and any $\alpha>0$ we may write 
\begin{equation}
\begin{aligned} & h(z)=\sum_{n=0}^{\infty}q^{a_{0}+a_{1}n+\cdots+a_{k}n^{k}}b_{n}^{it}z^{n}=\sum_{n=0}^{\infty}h_{n}q^{\alpha n^{2}}z^{n}\\
 & =\sum_{n=0}^{\infty}q^{a_{0}+a_{1}n+(a_{2}-\alpha)n^{2}+\cdots+a_{k}n^{k}}b_{n}^{it}q^{\alpha n^{2}}z^{n},
\end{aligned}
\label{eq:2.39}
\end{equation}
then $\sum_{n=0}^{\infty}h_{n}q^{\alpha n^{2}}z^{n}$ satisfies the
conditions of Corollary \ref{cor:special-sums}. Therefore, the entire
function $h(z)$ has infinitely many nonzero roots $\left\{ \rho_{n}\right\} _{n\in\mathbb{N}}$,
and $\log\left|\rho_{n}\right|$ grow at least as fast as a linear
function in $n$. 
\end{example}

\section{Applications \label{sec:Applications}}

In this section we apply Theorem \ref{thm:series} to prove that both
the entire functions $\mathcal{E}_{q}(z;t)$ and $\mathcal{L}(z;\alpha,q)$
have infinitely many nonzero roots and their moduli grow at least
exponentially. In the proofs we need the $q$-binomial theorem, \cite{AndrewsAskeyRoy,DLMF,IsmailBook,IsmailStanton}
\begin{equation}
\frac{(az;q)_{\infty}}{(z;q)_{\infty}}=\sum_{n=0}^{\infty}\frac{(a;q)_{n}}{(q;q)_{n}}z^{n},\label{eq:3.1}
\end{equation}
 where $a,z\in\mathbb{C}$ and $\left|z\right|<1$.

\subsection{Application to $\mathcal{E}_{q}(z;t)$}

In this subsection we assume $0<q<1$.
\begin{cor}
\label{thm:q-wave} For $0<q<1$ and $\left|t\right|<1$ the $q$-exponential
function $\mathcal{E}_{q}\left(z;t\right)$ has infinitely many nonzero
roots $\left\{ \rho_{n}(q,t)\right\} _{n=1}^{\infty}$ such that 
\begin{equation}
0<\left|\rho_{1}(q,t)\right|\le\left|\rho_{2}(q,t)\right|\le\cdots.\label{eq:3.2}
\end{equation}
For any $0<\eta<1$ and $C$ satisfying
\begin{equation}
C>\frac{1}{\log q^{-1}}>0,\label{eq:3.3}
\end{equation}
there exists a positive integer $N$ such that
\begin{equation}
\left|\rho_{n}(q,t)\right|\ge e^{C(\eta,C)n},\quad\forall n\ge N.\label{eq:3.4}
\end{equation}
Furthermore,\textup{ we have}
\begin{equation}
\mathcal{E}_{q}\left(\cos\theta;t\right)\neq0,\quad\left|t\right|<\frac{1-q}{2\left(1+q^{1/4}e^{-\Im\theta}\right)\left(1+q^{1/4}e^{\Im\theta}\right)}.\label{eq:3.5}
\end{equation}
 Similarly, for any $0<q<1$ and $t\in\mathbb{R}$ by \cite[14.1.15]{IsmailBook}
we have 
\begin{equation}
\Re\mathcal{E}_{q}\left(\cos\theta;it\right)>0,\quad\left|t\right|<\sqrt{\frac{(1-q)(1-q^{2})}{2q\left(1+e^{2\Im\theta}\right)\left(1+e^{-2\Im\theta}\right)}}\label{eq:3.6}
\end{equation}
 and 
\begin{equation}
\frac{(1-q)\Im\mathcal{E}_{q}\left(\cos\theta;it\right)}{2tq^{1/4}\cos\theta}>0,\quad\left|t\right|<\sqrt{\frac{(1-q)(1-q^{3})}{2q\left(1+e^{2\Im\theta}\right)\left(1+e^{-2\Im\theta}\right)}}.\label{eq:3.7}
\end{equation}
\end{cor}

\begin{proof}
The first part of this theorem are just consequences of Theorem \ref{thm:series}
and (\ref{eq:1.20}). We only to show (\ref{eq:3.5}). Assertions
(\ref{eq:3.6}) and (\ref{eq:3.7}) can be proved similarly.

It is known that for $\left|t\right|<1$ and $\theta\in[0,\pi]$,
\cite[Theorem 14.1.2]{IsmailBook}
\begin{equation}
\frac{\left(qt^{2};q^{2}\right)_{\infty}}{\left(-t;q^{1/2}\right)_{\infty}}\mathcal{E}_{q}\left(\cos\theta;t\right)={}_{2}\phi_{1}\left(\begin{array}{c}
q^{1/4}e^{i\theta},q^{1/4}e^{-i\theta}\\
-q^{1/2}
\end{array}\bigg|q^{1/2},-t\right).\label{eq:3.8}
\end{equation}
For any $n\in\mathbb{N}$ let
\begin{equation}
c_{n}=\frac{\left(q^{1/4}e^{i\theta},q^{1/4}e^{-i\theta};q^{1/2}\right)_{n}}{\left(-q^{1/2},q^{1/2};q^{1/2}\right)_{n}}(-t)^{n},\label{eq:3.9}
\end{equation}
then,
\begin{equation}
\begin{aligned} & \left|\frac{c_{n}}{c_{n-1}}\right|\le\left|\frac{\left(1-q^{n/2-1/4}e^{i\theta}\right)\left(1-q^{n/2-1/4}e^{-i\theta}\right)}{1-q^{n}}t\right|\\
 & \le\frac{\left(1+q^{1/4}e^{\Im\theta}\right)\left(1+q^{1/4}e^{-\Im\theta}\right)}{1-q}\left|t\right|\le\left|t\right|\frac{\left(1+q^{1/4}\right)\left(1+q^{1/4}e^{\left|\Im\theta\right|}\right)}{1-q}.
\end{aligned}
\label{eq:3.10}
\end{equation}
Clearly, for any $\delta>0$ and any $t$ satisfying
\begin{equation}
\left|t\right|<\frac{1-q}{\left(1+q^{1/4}\right)\left(1+q^{1/4}e^{\delta}\right)}<1,\label{eq:3.11}
\end{equation}
the infinite series on the right hand side of (\ref{eq:3.8}) converges
uniformly in $\left|\Im\theta\right|\le\delta$. Then it defines a
analytic function in variable $\theta$ inside $\left\{ \theta:\left|\Im\theta\right|<\delta\right\} $.
On the other hand it is known that for any $\left|t\right|<1$ the
function on the left hand side of (\ref{eq:3.8}),
\begin{equation}
\frac{\left(qt^{2};q^{2}\right)_{\infty}}{\left(-t;q^{1/2}\right)_{\infty}}\mathcal{E}_{q}\left(z;t\right)\label{eq:3.12}
\end{equation}
is analytic in z, and they agrees on $z\in[-1,1]$, then they must
be the same in $\left\{ \theta:\left|\Im\theta\right|<\delta\right\} $
by analytic continuation.

By (\ref{eq:3.10}) we have
\begin{equation}
\left|c_{n}\right|\le\left(\left|t\right|\frac{\left(1+q^{1/4}e^{-\Im\theta}\right)\left(1+q^{1/4}e^{\Im\theta}\right)}{1-q}\right)^{n},\quad\forall n\in\mathbb{N}.\label{eq:3.13}
\end{equation}
If 
\begin{equation}
\left|t\right|<\frac{1-q}{2\left(1+q^{1/4}e^{-\Im\theta}\right)\left(1+q^{1/4}e^{\Im\theta}\right)},\label{eq:3.14}
\end{equation}
then 
\begin{equation}
\begin{aligned} & \left|\frac{\left(qt^{2};q^{2}\right)_{\infty}}{\left(-t;q^{1/2}\right)_{\infty}}\mathcal{E}_{q}\left(\cos\theta;t\right)\right|\ge1-\left|\sum_{n=1}^{\infty}c_{n}\right|\\
 & \ge1-\sum_{n=1}^{\infty}\left(\left|t\right|\frac{\left(1+q^{1/4}e^{-\Im\theta}\right)\left(1+q^{1/4}e^{\Im\theta}\right)}{1-q}\right)^{n}>1-\sum_{n=1}^{\infty}\frac{1}{2^{n}}=0.
\end{aligned}
\label{eq:3.15}
\end{equation}
 
\end{proof}
For any $\alpha>0$ by \cite[(5.52)]{IsmailRZhang2}
\begin{equation}
\begin{aligned} & q^{\beta^{2}/2}\left(q^{\alpha+\beta+n+\frac{1}{2}};q\right)_{\infty}L_{n}^{\left(\alpha+\beta-\frac{1}{2}\right)}\left(x;q\right)\\
 & =\sqrt{\frac{\log q^{-1}}{2\pi}}\int_{-\infty}^{\infty}\frac{\left(xq^{\alpha+iy};q\right)_{n}q^{y^{2}/2+i\beta y}}{\left(q;q\right)_{n}\left(-q^{\alpha+iy};q\right)_{\infty}}dy
\end{aligned}
\label{eq:3.21}
\end{equation}
 and \cite[(5.30)]{IsmailRZhang2} 
\begin{equation}
q^{\alpha^{2}}A_{q}\left(q^{2\alpha}z\right)=\sqrt{\frac{\log q^{-1}}{4\pi}}\int_{-\infty}^{\infty}\frac{q^{y^{2}/4+i\alpha y}}{\left(-ze^{iy};q\right)_{\infty}}dy,\label{eq:3.22}
\end{equation}
we have

\begin{equation}
\begin{aligned} & \mathcal{L}(z;\alpha,q)=\sum_{n=0}^{\infty}\left(q^{\alpha+2n+\frac{1}{2}};q\right)_{\infty}L_{n}^{\left(\alpha+n-\frac{1}{2}\right)}\left(z;q\right)q^{n^{2}/2}q^{\alpha n}\\
 & =\sqrt{\frac{\log q^{-1}}{2\pi}}\sum_{n=0}^{\infty}\int_{-\infty}^{\infty}\frac{\left(zq^{\alpha+iy};q\right)_{n}q^{y^{2}/2}q^{(\alpha+iy)n}}{\left(q;q\right)_{n}\left(-q^{\alpha+iy};q\right)_{\infty}}dy=\sqrt{\frac{\log q^{-1}}{2\pi}}\int_{-\infty}^{\infty}\frac{\left(zq^{2\alpha+2iy};q\right)_{\infty}q^{y^{2}/2}dy}{\left(q^{\alpha+iy},-q^{\alpha+iy};q\right)_{\infty}}\\
 & =\sqrt{\frac{\log q^{-1}}{2\pi}}\int_{-\infty}^{\infty}\frac{\left(zq^{2\alpha+2iy};q\right)_{\infty}q^{y^{2}/2}dy}{\left(q^{2\alpha+2iy};q^{2}\right)_{\infty}}=\sqrt{\frac{\log q^{-1}}{2\pi}}\sum_{n=0}^{\infty}\frac{\left(-zq^{2\alpha-1/2}\right)^{n}q^{n^{2}/2}}{(q;q)_{n}}\int_{-\infty}^{\infty}\frac{q^{2niy+y^{2}/2}dy}{\left(q^{2\alpha+2iy};q^{2}\right)_{\infty}}\\
 & =\sum_{n=0}^{\infty}\frac{\left(-zq^{2\alpha-1/2}\right)^{n}q^{5n^{2}/2}}{(q;q)_{n}}A_{q^{2}}\left(-q^{4n+2\alpha}\right),
\end{aligned}
\label{eq:3.23}
\end{equation}
which gives 
\begin{equation}
\begin{aligned} & \sum_{n=0}^{\infty}\left(q^{\alpha+2n+\frac{1}{2}};q\right)_{\infty}L_{n}^{\left(\alpha+n-\frac{1}{2}\right)}\left(z;q\right)q^{n^{2}/2}q^{\alpha n}\\
 & =\sum_{n=0}^{\infty}\frac{\left(-zq^{2\alpha-1/2}\right)^{n}q^{5n^{2}/2}}{(q;q)_{n}}A_{q^{2}}\left(-q^{4n+2\alpha}\right).
\end{aligned}
\label{eq:3.24}
\end{equation}
Clearly, for any $\alpha>0$ the infinite series $\mathcal{L}(z;\alpha,q)$
in (\ref{eq:1.16}) defines an entire function of $z\in\mathbb{C}$.
By Corollary \ref{cor:special-sums} $\mathcal{L}(z;\alpha,q)$ has
infinitely many zeros and their magnitude grows at least as fast as
$e^{cn}$ for some $c>0$.

\end{document}